\documentclass[a4paper, 12pt]{amsart}
\usepackage{mathrsfs}
\usepackage{amsmath}
\usepackage{amssymb}
\usepackage{verbatim}
\usepackage{xcolor}
\usepackage{CJK}
\usepackage{graphicx}
\usepackage[all]{xy}
\usepackage{appendix}
\DeclareGraphicsExtensions{.eps,.ps,.jpg,.bmp}

\theoremstyle{plain}

\newtheorem{theorem}{Theorem}[section]
\newtheorem{proposition}{Proposition}[section]
\newtheorem{corollary}{Corollary}[section]

\newtheorem{lemma}{Lemma}[section]

\newtheorem{problem}{Problem}[section]
\newtheorem{example}{Example}[section]

\theoremstyle{definition}
\newtheorem{definition}{Definition}[section]

\theoremstyle{remark}

\newtheorem*{acknowledgment}{Acknowledgment}

\title{An alternative approach on the singular metric on a vector bundle}
\author{Jingcao Wu}

\begin{document}
\pagestyle{plain}
\begin{abstract}
In this paper, we will provide an alternative definition for the singular Hermitian metric on a vector bundle. Moreover, we discuss the Griffiths and Nakano positivities under this circumstance and prove a generalised Griffiths' vanishing theorem.
\end{abstract}
\maketitle

\section{Introduction}
The notion of the singular Hermitian metric was first developed on a holomorphic line bundle, and it turns out to be a powerful tool in complex geometry. There are fruitful work in these aspects, such as \cite{Dem92a,Dem92b,Dem12,DPS01,Nad90,Siu98}. It is natural to consider the singular metric on a vector bundle of higher rank. The pioneering work dates back to \cite{Cat98}. It first introduced the singular metric on a vector bundle, and used it to treat the global generation problems. We call it Cataldo's singular metric in this paper for simplicity. After that, a different definition for the singular metric was given in \cite{BP08,Rau15}. This type of singular metric was also widely studied in recent years, and we will call it Raufi's singular metric.

The most important problem when concerning a singular metric on a vector bundle is how to deal with the associated curvature. Different from the situation on a line bundle, the example (Theorem 1.5) in \cite{Rau15} shows that the curvature associated with a singular metric on a vector bundle may not be a current with measure coefficients.

Therefore it will cause problems when talking about the positivity such as the Griffiths and Nakano positivities. In traditional Hermitian geometry, the definitions of these two positivities are deeply relied on the curvature \cite{Gri69}. In \cite{BP08,Rau15}, a clever way is used to approach the Griffiths and Nakano semi-positivity without involving curvature. However, it seems hard to reapply this strategy when dealing with the strict Griffiths positivity or Nakano positivity. There are several papers in this aspect, such as \cite{Ina20b,Rau15}, in which these notions are defined under certain restrictions.

The main goal of this paper is to discuss the positivity of a vector bundle with singular metrics. In particular, we will define the (strict) Griffiths positivity and Nakano positivity in the most general context. In order to do this, we first provide an alternative way to describe the singular metric. More precisely, for a given vector bundle $E$ of rank $r+1$ over a complex manifold $Y$, we consider its projectivsed bundle $X:=\mathbb{P}(E^{\ast})$ as well as the tautological line bundle $\mathcal{O}_{E}(1):=\mathcal{O}_{X}(1)$. Let $\pi:X\rightarrow Y$ be the natural projection.

A basic observation in Finsler geometry says that a complex Finsler metric $G$ on $E$, which we will explain later, is one-one corresponds to a smooth Hermitian metric $\varphi$ on $\mathcal{O}_{E}(1)$. The explicit correspondence is presented in Sect.2. On the other hand, we can canonically define an $L^{2}$-metric $H_{\varphi}$ on $E$ through $\varphi$. In particular, $H_{\varphi}=G$ up to a constant if and only if $G$ is itself a Hermitian metric \cite{LSY06}.

So it is reasonable to define the singular Hermitian metric on $E$ via the singular metric on $\mathcal{O}_{E}(1)$. Throughout this paper, we always consider those locally $L^{1}$-bounded functions $\varphi$ on $X$ such that
\[
Y_{\varphi}:=\{z\in Y;\varphi|_{X_{z}}\textrm{ is well-defined}\}
\]
is a dense subset of $Y$. Now we fix a smooth Hermitian metric $h_{0}$ on $\mathcal{O}_{E}(1)$ as the reference metric. Let
\[
\begin{split}
\mathcal{H}(X):=&\{\varphi\in L^{1}_{\textrm{loc}}(X);\textrm{on each }X_{z}\textrm{ with }z\in Y_{\varphi}, \varphi|_{X_{z}}\textrm{ is smooth and }\\
&(i\Theta_{\mathcal{O}_{X}(1),h_{0}}+i\partial\bar{\partial}\varphi)|_{X_{z}}\textrm{ is strictly positive}\}.
\end{split}
\]
Thus by definition, if $\varphi\in\mathcal{H}(X)$ and $z\in Y_{\varphi}$,
\[
\omega_{\varphi,z}:=(i\Theta_{\mathcal{O}_{X}(1),h_{0}}+i\partial\bar{\partial}\varphi)|_{X_{z}}
\]
is actually a K\"{a}her metric on $X_{z}$. Let $u$ be a (local) section of $E$, then we have the following definition.
\begin{definition}\label{d11}
A singular Hermitian metric on $E$ is a map $H_{\varphi}$ with form that
\[
H_{\varphi}(u,u)=\int_{X_{z}}|u|^{2}_{h_{0}}e^{-\varphi}\frac{\omega^{r}_{\varphi,z}}{r!},
\]
where $\varphi\in\mathcal{H}(X)$. Here we use the fact that $\pi_{\ast}\mathcal{O}_{E}(1)=E$.
\end{definition}
Obviously, $H_{\varphi}(u,u)$ is finite on $Y_{\varphi}$. We remark here that $H_{\varphi}(u,u)$ may not be locally integrable since $\varphi\in L^{1}_{\textrm{loc}}$ won't guarantee that $e^{-\varphi}\in L^{1}_{\textrm{loc}}$. One could easily find such a counterexample. However, if those unmeasurable ones are excluded, we will return back to Ruafi's singular metric.

More specifically, let $M(E)$ be the set of all the smooth Hermitian metrics on $E$, let $M_{R}(E)$ be the set of Raufi's singular Hermitian metrics, and let $M_{W}(E)$ be the set of all the singular Hermitian metrics in Definition \ref{d11}. We will show that
\begin{theorem}\label{t1}
Let $Y$ be a complex manifold, and let $E$ be a holomorphic vector bundle over $Y$. Then we have
\begin{enumerate}
  \item $M(E)\subset M_{W}(E)$.
  \item When $\textrm{rank} E=1$, $M_{W}(E)$ equals to the set of all the singular metrics on a line bundle defined in \cite{Dem12}.
  \item Let $M^{o}_{R}(E)$ be the set of Raufi's singular Hermitian metrics $H$, such that $H$ is non-degenerated (as a Hermitian form) on each fibre. Let $M^{o}_{W}(E)$ be the set of all the singular Hermitian metrics $H$ defined above, such that $H$ is measurable. Then $M^{o}_{W}(E)=M^{o}_{R}(E)$.
\end{enumerate}
\end{theorem}
Therefore our definition for the singular metric is a reformulation of Raufi's. This reformulation is necessary and plays a crucial role in the later part of our work.

In the rest part, $Y$ is moreover assumed to be compact. We are ready to talk about the positivity involving the singular metric on a vector bundle $E$. Fix the notations as before. Observe that for a given singular metric $H_{\varphi}\in M_{W}(E)$, it naturally induces a (singular) metric $\psi$ on $\mathcal{O}_{E}(1)$. One should pay attention that $\psi$ is not necessary equal to $\varphi$, but we surely have that $H_{\varphi}=H_{\psi}$ up to a constant due to Theorem \ref{t1}. Thus, we may assume that the corresponding metric on $\mathcal{O}_{E}(1)$ of $H_{\varphi}$ is $\varphi$ itself for the time being. Then we have the following definition.
\begin{definition}\label{d12}
Let $(E,H_{\varphi})$ be a (singular) Hermitian vector bundle over $Y$. We say $(E,H_{\varphi})$ is positive in the sense of Griffiths, if
\[
i\Theta_{\mathcal{O}_{E}(1),\varphi}\geqslant0
\]
in the sense of current. $(E,H_{\varphi})$ is strictly positive in the sense of Griffiths, if
\[
i\Theta_{\mathcal{O}_{E}(1),\varphi}\geqslant\omega
\]
for a Hermitian metric $\omega$ on $X$.
\end{definition}
Notice that $\mathcal{O}_{E}(1)$ is a line bundle, and the curvature current associated with $\varphi$ is well-understood. We will show (in Theorem \ref{t3}) that this definition is compatible with the definition of Griffiths positivity in \cite{Gri69} for a smooth Hermitian metric. Moreover, Theorem \ref{t3} together with the discussions in \cite{PTa18} implies that if $(E,H_{\varphi})$ is positive in the sense of Griffiths, $H_{\varphi}$ must be measurable.

However, before proceeding to Theorem \ref{t3}, we would like to present an equivalent description for the Griffiths positivity, which is of independent interest. Let
\[
\mathcal{H}_{i}(X):=\{\varphi\in\mathcal{H}(X);\varphi\textrm{ is induced by a singular metric }H\in M_{W}(E)\}.
\]
Then we have
\begin{theorem}\label{t2}
Consider the $\varphi\in\mathcal{H}(X)$ such that
\[
i\Theta_{\mathcal{O}_{E}(1),h_{0}}+i\partial\bar{\partial}\varphi\geqslant0.
\]
The following two statements are equivalent.
\begin{enumerate}
  \item[(a)] $\varphi\in\mathcal{H}_{i}(X)$
  \item[(b)] $\varphi$ induces an isometry between the canonical isomorphism
  \begin{equation}\label{e11}
  K_{X/Y}\simeq\mathcal{O}_{X}(-r-1)+\pi^{\ast}\det E.
  \end{equation}
\end{enumerate}
\end{theorem}
We will make an explicit explanation of this isometry in the Sect.4. A result similar with Theorem \ref{t2} appears in \cite{Nau17}, and our theorem benefits a lot from it.

Remember that by definition, a singular metric on $E$ is given by a function $\varphi$ in $\mathcal{H}(X)$. However, it is not a one-one correspondence between $\mathcal{H}(X)$ and $M_{W}(E)$.  Theorem \ref{t2} actually gives a subset of $\mathcal{H}(X)$ that bijectively maps to $M_{W}(E)$. Indeed let $\mathcal{H}_{h_{0}}(X)$ be the set of all the functions $\varphi\in\mathcal{H}(X)$ with the following properties:
\begin{enumerate}
  \item[(i)] $i\Theta_{\mathcal{O}_{E}(1),h_{0}}+i\partial\bar{\partial}\varphi\geqslant0$ in the sense of current;
  \item[(ii)] For $z\in Y_{\varphi}$,  $\varphi$ induces an isometry between
\[
K_{X/Y}\simeq\mathcal{O}_{X}(-r-1)+\pi^{\ast}\det E.
\]
\end{enumerate}

Then as a consequence of Theorem \ref{t2}, we have
\begin{corollary}\label{c11}
Let $E$ be a vector bundle. Then there is a one-one correspondence between the set of the (singular) Hermitian metrics $H$ on $E$ such that $(E,H)$ is Griffiths positive in the sense of Definition \ref{d12} and $\mathcal{H}_{h_{0}}(X)$.
\end{corollary}
From now on, for an arbitrary $H_{\varphi}\in H_{W}(E)$ we can always assume that the corresponding metric on $\mathcal{O}_{E}(1)$ is $\varphi$ again without loss of generality. In particular, $(E,H_{\varphi})$ is (strictly) positive in the sense of Griffiths if and only if $(\mathcal{O}_{E}(1),\varphi)$ is (big) pseudo-effective.

Next we discuss the Nakano positivity. Remember in Definition 1.8 of \cite{Rau15}, there has already been a clever approach. Unfortunately, $(E,H_{\varphi})$ is a priori supposed to be positive in the sense of Griffiths there. We provide an alternative notion called strong Nakano positivity here. In particular, we do not need to assume that $E$ is positive in the sense of Griffiths at the beginning.

To begin with, we generalise Definition \ref{d12} to a $\mathbb{Q}$-twisted vector bundle $E<\delta>$ with $\delta$ a $\mathbb{Q}$-divisor on $Y$. The definition for a $\mathbb{Q}$-twisted vector bundle will be given in Sect.2. Now let $(E,H)$ be a (singular) Hermitian vector bundle of rank $r+1$, and let
\[
\delta=q\det E
\]
be a $\mathbb{Q}$-divisor with $q\in\mathbb{Q}$. Let $\varphi$ be the metric on $\mathcal{O}_{E}(1)$ corresponding to $H$. We say
\[
(E<q\det E>,H)
\]
is (strictly) positive in the sense of Griffiths if
\[
c_{1}(\mathcal{O}_{E}(1),\varphi)+q\pi^{\ast}c_{1}(\det E,\det H)
\]
is (strictly) positive in the sense of current. Remember that
\[
 \pi:X=\mathbb{P}(E^{\ast})\rightarrow Y
\]
is the projection.

Then the strong Nakano positivity is defined as
\begin{definition}\label{d13}
Let $(E,H)$ be a (singular) Hermitian vector bundle. Then it is (strictly) strongly positive in the sense of Nakano if the $\mathbb{Q}$-twisted vector bundle
\[
(E<-\frac{1}{r+2}\det E>,H)
\]
is (strictly) positive in the sense of Griffiths.

Let $q$ be an arbitrary rational number. Then
\[
(E<q\det E>,H)
\]
is (strictly) strongly positive in the sense of Nakano if $\mathbb{Q}$-twisted vector bundle
\[
(E<\frac{2q-1}{r+2}\det E>,H)
\]
is (strictly) positive in the sense of Griffiths.
\end{definition}

In the available literature, such as \cite{Ina20c,Rau15}, before talking about the Nakano positivity, $(E,H)$ should a priori be Griffiths positive. That is one motivation for us to consider this strong Nakano positivity, which is independently defined. We summarise the discussions about the Griffiths and Nakano positivities as the following theorem.
\begin{theorem}\label{t3}
Let $Y$ be a compact complex manifold, and let $E$ be a vector bundle over $Y$ endowed with a (singular) Hermitian metric $H$. Then
\begin{enumerate}
  \item When $H$ is smooth, the notion of Griffiths  positivity here coincides with the definition in \cite{Gri69}. The strong Nakano positivity will implies the Nakano positivity in the usual sense.
  \item $(E,H)$ is positive (negative) in the sense of Griffiths if and only if it is positively (negatively) curved in the sense of Definition 1.2 in \cite{Rau15}.
  \item If $(E,H)$ is (strictly) positive in the sense of Griffiths, then
  \[
  (\det E,\det H)
  \]
  is (big) pseudo-effective.
  \item If $(E,H)$ is strongly positive in the sense of Nakano, then it is positive in the sense of Griffiths.
  \item $(E,H)$ is positive in the sense of Griffiths if and only if
\[
(E<\frac{1}{2}\det E>,H)
\]
is strongly positive in the sense of Nakano.
  \item If $(E,H)$ is strongly negative in the sense of Nakano, then it is negatively curved in the sense of Nakano as in Definition 1.8 of \cite{Rau15}.
\end{enumerate}
\end{theorem}
We have seen in Theorem \ref{t3} that the strong Nakano positivity implies the Nakano positivity, but we are not sure whether the opposite direction holds. Finally we generalise Griffiths' vanishing theorem in \cite{Gri69}.
\begin{theorem}\label{t4}
Let $Y$ be a compact K\"{a}hler manifold. Let $(E,H)$ be a (singular) Hermitian vector bundle of rank $r+1$ and let $L$ be a line bundle over $Y$. Then we have the following results.
\begin{enumerate}
  \item Suppose that $(E,H)$ is strictly Griffiths positive, and $L$ is nef. Let $\varphi$ be the corresponding metric on $\mathcal{O}_{E}(1)$. For any $k\geqslant0$, if
      \[\mathscr{I}((r+k+1)\varphi)=\mathcal{O}_{X},
      \]
then
  \[
  H^{q}(Y,K_{Y}\otimes S^{k}E\otimes\det E\otimes L)=0\textrm{ for }q>0.
  \]
  \item Suppose that $(E,H)$ is Griffihts positive, and $(L,h)$ is big. Let $\varphi$ be the corresponding metric on $\mathcal{O}_{E}(1)$. Assume that $\nu(\varphi)=0$. Then
  \[
  H^{q}(Y,K_{Y}\otimes S^{k}E\otimes\det E\otimes L\otimes\mathscr{I}(h))=0
  \]
  for $q>0$ and $k\geqslant0$.
  \item Suppose that $(E,H)$ is strictly strongly Nakano positive, and $L$ is nef. Let $\varphi$ be the corresponding metric on $\mathcal{O}_{E}(1)$. Assume that $\nu(\varphi)=0$.
Then
  \[
  H^{q}(Y,K_{Y}\otimes S^{k}E\otimes L\otimes\mathscr{I}(\det H))=0
  \]
for $q>0$ and $k\geqslant1$.
  \item Suppose that $(E,H)$ is strongly Nakano positive, and $(L,h)$ is big. Let $\varphi$ be the corresponding metric on $\mathcal{O}_{E}(1)$. Assume that $\nu(\varphi)=0$. Then
  \[
  H^{q}(Y,K_{Y}\otimes S^{k}E\otimes L\otimes\mathscr{I}(\det H\otimes h))=0
  \]
  for $q>0$ and $k\geqslant0$.
\end{enumerate}
\end{theorem}
Here $\mathscr{I}(\varphi),\mathscr{I}(h)$ and $\mathscr{I}(\det H)$ refer to the multiplier ideal sheaves. In \cite{Ina20b}, there is also a generalisation of Griffiths' vanishing theorem (Corollary 1.4). We will use an example in the end of this paper to show the partial relationship between his result and Theorem \ref{t4} here.

\begin{acknowledgment}
The author wants to thank Prof. Jixiang Fu, who brought this problem to his attention and for numerous discussions directly related to this work.
\end{acknowledgment}

\section{Preliminary}
\subsection{Setup}
Let $Y$ be a complex manifold of dimension $n$, and let $f:E\rightarrow Y$ be a holomorphic vector bundle of rank $r+1$ over $Y$. Let $z=(z_{1},...,z_{n})$ be a system of local coordinate on $Y$, and let
\[
Z=(Z_{0},...,Z_{r})
\]
be the fibre coordinate defined by a local holomorphic frame $\{u_{0},...,u_{r}\}$ of $E$. $X$ is defined as the total space of the projectivised bundle $\mathbb{P}(E^{\ast})$. Let $(z,w=(w_{1},...,w_{r}))$ be a system of local coordinate on $X$, and let $\pi:X\rightarrow Y$ be the projection. Let $\mathcal{O}_{E}(1):=\mathcal{O}_{X}(1)$ be the tautological line bundle.

Throughout this paper, by saying an $L^{1}_{\textrm{loc}}$-bounded function $\varphi$ on $X$, we always refer to those $\varphi$ whose singular part doesn't dominate $Y$. In other word, the closure of $\{z\in Y;\varphi|_{X_{z}}=\infty\}$ is a proper subset of $Y$.
\subsection{Hermitian geometry revisit}
Let's briefly recall the canonical Hermitian geometry first. This part is mostly taken from \cite{Gri69,Kob87}. Let $H$ be a smooth Hermitian metric on $E$, we write
\[
H_{i}=\frac{\partial H}{\partial z_{i}}, H_{\bar{j}}=\frac{\partial H}{\partial\bar{z}_{j}}, H_{\alpha}=\frac{\partial H}{\partial Z_{\alpha}}, H_{\bar{\beta}}=\frac{\partial H}{\partial\bar{Z}_{\beta}}
\]
to denote the derivative with respect to
\[
z_{i},\bar{z}_{j} (1\leqslant i,j\leqslant n)\textrm{ and } Z_{\alpha},\bar{Z}_{\beta}  (0\leqslant\alpha,\beta\leqslant r).
\]
The higher order derivative is similar. Then locally $H$ is represented as the matrix
\[
(H_{\alpha\bar{\beta}}(z)).
\]
Since $(H_{\alpha\bar{\beta}})$ is positive-definite, it is invertible. The inverse matrix is denoted by $(H^{\gamma\bar{\delta}})$. In this context, the associated curvature is represented as
\[
\Theta_{E,H}=\sum\Theta_{\alpha\bar{\beta}i\bar{j}}dz^{i}\wedge d\bar{z}^{j}\otimes Z_{\alpha}\otimes Z^{\ast}_{\beta}
\]
with
\[
\Theta_{\alpha\bar{\beta}i\bar{j}}=-H_{\alpha\bar{\beta}i\bar{j}}+H^{\gamma\bar{\delta}}H_{\alpha\bar{\delta}i}H_{\gamma\bar{\beta}\bar{j}}.
\]
Now fix a point $z_{0}\in Y$, we can always assume that $\{u_{0},...,u_{r}\}$ is an orthonormal basis with respect to $H$ at $z_{0}$. The Griffiths and Nakano positivities \cite{Gri69} is defined as follows:
\begin{definition}\label{d21}
Keep notations before,
\begin{enumerate}
  \item $E$ is called (strictly) positive in the sense of Griffiths at $z_{0}$, if for any complex vector $z=(z_{1},...,z_{n})$ and section $u=\sum Z_{\alpha}u_{\alpha}$ of $E$,
  \[
  \sum i\Theta_{\alpha\bar{\beta}i\bar{j}}Z_{\alpha}\bar{Z}_{\beta}z_{i}\bar{z}_{j}
  \]
  is (strictly) positive.
  \item $E$ is called (strictly) positive in the sense of Nakano at $z_{0}$, if for any $n$-tuple $(u^{1}=\sum Z^{1}_{\alpha}u_{\alpha},...,u^{n}=\sum Z^{n}_{\alpha}u_{\alpha})$ of sections of $E$,
  \[
  \sum i\Theta_{\alpha\bar{\beta}i\bar{j}}Z^{i}_{\alpha}\bar{Z}^{j}_{\beta}
  \]
  is (strictly) positive.
\end{enumerate}
\end{definition}

Let $u$ be a holomorphic section of $E$. A direct computation implies the following well-known formula in \cite{Gri69}
\begin{equation}\label{e21}
i\partial\bar{\partial}\log H(u)=-i\frac{\langle\Theta u,u\rangle}{H(u)}+i\frac{\langle D^{1,0}u,D^{1,0}u\rangle H(u)-\langle u,D^{1,0}u\rangle\langle D^{1,0}u,u\rangle}{H(u)^{2}}.
\end{equation}
Note that the second term on the right hand side is positive by the Cauchy--Schwarz inequality. Then we have
\begin{lemma}\label{l21}
$(E,H)$ is negative in the sense of Griffiths if and only if $\log H(u)$ is plurisubharmonic for any holomorphic section $u$.
\end{lemma}
One should pay attention that there is no equivalence between the strict Griffiths positivity of $(E,H)$ and the strict plurisubharmonicity of $\log H(u)$.

The famous Griffiths' vanishing theorem says that
\begin{proposition}[Griffiths' vanishing theorem, \cite{Gri69}]\label{p21}
Let $Y$ be a projective manifold, and let $E$ be a strictly Griffiths positive vector bundle over $Y$. Then
\[
H^{q}(Y,K_{Y}\otimes E\otimes\det E)=0\textrm{ for }q>0.
\]
\end{proposition}
This theorem can be seen as a consequence of Nakano's vanishing theorem \cite{Nak55} in view of Demailly and Skoda's work \cite{DSk79}. This observation directly leads to our generalisation, i.e. Theorem \ref{t4}.
\begin{proof}[A quick proof of Proposition \ref{p21}]
Since $E$ is positive in the sense of Griffiths, $E\otimes\det E$ is positive in the sense of Nakano by \cite{DSk79}. Therefore
\[
H^{q}(Y,K_{Y}\otimes E\otimes\det E)=0\textrm{ for }q>0
\]
by Nakano's vanishing theorem.
\end{proof}

Proposition \ref{p21} was further developed by \cite{Man97} as
\begin{proposition}[Manivel's vanishing theorem]\label{p22}
Let $Y$ be a projective manifold. Let $E$ be a vector bundle of rank $r+1$ and let $L$ be a line bundle over $Y$. Suppose
that $E$ is ample and $L$ nef; or that $E$ is nef and $L$ ample. Then for any $k\geqslant 0$,
\[
H^{p,q}(Y, S^{k}E\otimes(\det E)^{n-p+1}\otimes L)=0\textrm{ for }p+q>n.
\]
\end{proposition}
One refers Definition \ref{d23} for the definition of an ample (nef) vector bundle. Theorem \ref{t4} also extends Proposition \ref{p22} (with $p=n$) to the singular case.

\subsection{Finsler geometry revisit}
We present some basic properties concerning the Finsler metric on a vector bundle here. This part is mainly based on \cite{Kob75,Kob96,Laz04}. Keep the same notations as above, we have
\begin{definition}\label{d22}
A Finsler metric $G$ on $E$ is a continuous function $G:E\rightarrow\mathbb{R}$ satisfying the following conditions:
\begin{enumerate}
  \item $G$ is smooth on $E^{0}:=E-\{0\}$, where $0$ denotes the zero section of $E$;
  \item $G(z,Z)\geqslant 0$ for all $(z,Z)\in E$ with $z\in Y$ and $Z\in f^{-1}(z)$, and $G(z,Z)=0$ if and only if $Z=0$;
  \item $G(z,\lambda Z)=|\lambda|^{2}G(z,Z)$ for all $\lambda\in\mathbb{C}$.
\end{enumerate}
In applications, one often requires that $G$ is strongly pseudo-convex, that is,

(4) the Levi form $i\partial\bar{\partial}G$ on $E^{0}$ is positive-definite along fibres $E_{z}$ for $z\in Y$.
\end{definition}

Now let $G$ be a strongly pseudo-convex Finsler metric. Locally $G$ can be rewritten as
\[
G=\sum iG_{\alpha\bar{\beta}}(z,Z)dZ_{\alpha}\wedge d\bar{Z}_{\beta}.
\]
Different from the Hermitian metric, $G_{\alpha\bar{\beta}}$ may depend on the variable $Z$. It is easy to verify that $G_{\alpha\bar{\beta}}(z,Z)$ is homogenous of degree $0$ with respect to $Z$. Since $G$ is strongly pseudo-convex, $\{G_{\alpha\bar{\beta}}(z,Z)\}$ is a positive-definite matrix on $E^{0}$, hence a smooth metric. It then defines a smooth Hermitian metric $g^{G}$ on $p:f^{\ast}E\rightarrow E^{0}$ by
\[
g^{G}(z,Z,W):=G\circ p(z,Z,W)=G(z,Z).
\]
Here $W$ is the coordinate of the fibre of $p$. As is pointed before, $G_{\alpha\bar{\beta}}$ is homogenous of degree $0$ with respect to $Z$, hence it lifts to a Hermitian metric on $f^{\ast}E\rightarrow\mathbb{P}(E)$ over the projectivised bundle, which is still denoted by $g^{G}$. Then, as a subbundle of $f^{\ast}E$, the tautological line bundle
\[
\mathcal{O}_{\mathbb{P}(E)}(-1):=\{(z,[Z],W);W=\lambda Z,\lambda\in\mathbb{C}\}
\]
inherits a Hermitian metric from $(f^{\ast}E,g^{G})$, which is denoted by $g$. One verifies easily that $g(z,[Z],\xi)=G(z,Z)$ for any point $(z,[Z],\xi)$ on $\mathcal{O}_{\mathbb{P}(E)}(-1)$. Let $\varphi$ be the weight function of $g$, formally we have
\[
\varphi=\log G.
\]
Obviously, this procedure is invertible. In summary, there is a one-one correspondence between the Finsler metrics on $E$ and the smooth Hermitian metrics on $\mathcal{O}_{\mathbb{P}(E)}(-1)$. Therefore we call both of them the Finsler metric on $E$, if nothing is confused.

One should pay attention that $\mathcal{O}_{\mathbb{P}(E)}(1)$, as the dual bundle of
\[
\mathcal{O}_{\mathbb{P}(E)}(-1)\rightarrow\mathbb{P}(E),
\]
and $\mathcal{O}_{E}(1)\rightarrow\mathbb{P}(E^{\ast})$ are totally different line bundles, and in the rest part we will mainly work on $\mathcal{O}_{E}(1)$. In other words, we are more interested in the Finsler metric on $E^{\ast}$. One reason is that $\mathcal{O}_{E}(1)$ is usually used in \cite{Laz04} to define the algebraic positivities of $E$.
\begin{definition}\label{d23}
Keep notations before,
\begin{enumerate}
  \item $E$ is called ample, if $\mathcal{O}_{E}(1)$ is ample.
  \item $E$ is called nef, if $\mathcal{O}_{E}(1)$ is nef.
  \item $E$ is called big, if $\mathcal{O}_{E}(1)$ is big.
\end{enumerate}
\end{definition}
One refers to Sect.2.4 or \cite{Dem12} for the concepts of an ample (nef, big) line bundle.

Next, we establish the relationship between these algebraic positivities and the curvature of $E$. Remember that a Finsler metric $G$ on $E^{\ast}$ (not $E$) induces a Hermitian metric $\varphi$ on $\mathcal{O}_{E}(1)$. We should present a precise computation for the curvature $i\Theta_{\mathcal{O}_{E}(1),\varphi}$ here. This part is mainly based on \cite{Kob75,Kob96}. We expand $i\partial\bar{\partial}\varphi$ on $X$ as follows:
\[
i\partial\bar{\partial}\varphi=\sum i(g_{i\bar{j}}dz_{i}\wedge d\bar{z}_{j}+g_{i\bar{\beta}}dz_{i}\wedge d\bar{w}_{\beta}+g_{\alpha\bar{j}}dw_{\alpha}\wedge d\bar{z}_{j}+g_{\alpha\bar{\beta}}dw_{\alpha}\wedge d\bar{w}_{\beta}).
\]
Since $G$ is strongly pseudo-convex, $\{G_{\alpha\bar{\beta}}\}$ as well as $\{(\log G)_{\alpha\bar{\beta}}\}$ is invertible. Let $\{G^{\bar{\beta}\alpha}\}$ and $\{(\log G)^{\bar{\beta}\alpha}\}$ be their inverse matrixes respectively. Then for the holomorphic vector field $\frac{\partial}{\partial z_{i}}$ on $Y$, its horizonal lift to $X$ is defined as
\[
\frac{\delta}{\delta z_{i}}:=\frac{\partial}{\partial z_{i}}-\sum_{\alpha,\beta}(\log G)^{\bar{\beta}\alpha}(\log G)_{\bar{\beta}i}\frac{\partial}{\partial w_{\alpha}}.
\]
The dual basis of $\{\frac{\delta}{\delta z_{i}},\frac{\partial}{\partial w_{\alpha}}\}$ will be
\[
\{dz_{i},\delta w_{\alpha}:=d w_{\alpha}+\sum_{i,\beta}(\log G)_{\bar{\beta}i}(\log G)^{\bar{\beta}\alpha}dz_{i}\}.
\]
Let
\[
\Psi:=\sum iK_{\alpha\bar{\beta}i\bar{j}}\frac{w_{\alpha}\bar{w}_{\beta}}{G}dz_{i}\wedge d\bar{z}_{j}
\]
and
\[
\omega_{FS}:=\sum i\frac{\partial^{2}\log G}{\partial w_{\alpha}\partial\bar{w}_{\beta}}\delta w_{\alpha}\wedge\delta\bar{w}_{\beta},
\]
where
\[
K_{\alpha\bar{\beta}i\bar{j}}:=-G_{\alpha\bar{\beta}i\bar{j}}+G^{\gamma\bar{\delta}}G_{\alpha\bar{\delta}i}G_{\gamma\bar{\beta}\bar{j}}.
\]
It is easy to verify that they are globally defined $(1,1)$-form on $X$. Then the celebrated theorem given by Kobayashi says that
\begin{proposition}[Kobayashi, \cite{Kob75,Kob96}]\label{p23}
\begin{equation}\label{e22}
i\partial\bar{\partial}\varphi=-\Psi+\omega_{FS}.
\end{equation}
\end{proposition}

Notice that the definition of $K_{\alpha\bar{\beta}i\bar{j}}$ is formally the same as the curvature $\Theta_{\alpha\bar{\beta}i\bar{j}}$ for a Hermitian metric in Sect.2.2. Indeed, it is easy to verify that if $G$ is itself a Hermitian metric, i.e. $G_{\alpha\bar{\beta}}$ is independent of $Z$, $K_{\alpha\bar{\beta}i\bar{j}}$ is just the curvature in the usual sense. The formula (\ref{e22}) is interpreted to the formula (\ref{e21}) at this time. So it is reasonable to consider $\Psi$ as the "curvature" associated with the Finsler metric $G$, which leads to the following two definitions.
\begin{definition}\label{d24}
The form $\Psi$ is called the Kobayashi curvature associated with $G$. $(E^{\ast},G)$ is called (strictly) positive in the sense of Kobayashi if the matrix $(K_{\alpha\bar{\beta}i\bar{j}}w_{\alpha}\bar{w}_{\beta})$ is (strictly) positive for any
\[
w=(w_{1},...,w_{r}).
\]
The Kobayashi negativity is similar.
\end{definition}
\begin{definition}[\cite{CTi08,Don01,Don05}]\label{d25}
\[
c(\varphi)_{i\bar{j}}:=<\frac{\delta}{\delta z_{i}},\frac{\delta}{\delta z_{j}}>_{i\partial\bar{\partial}\varphi}
\]
is called the geodesic curvature of $\varphi$ in the direction of $i,j$.
\end{definition}
Let $a^{\alpha}_{i}:=-\sum_{\beta}(\log G)^{\bar{\beta}\alpha}(\log G)_{\bar{\beta}i}$, we have
\[
c(\varphi)_{i\bar{j}}=(\log G)_{i\bar{j}}-\sum_{\alpha}a^{\alpha}_{i}a_{\alpha\bar{j}}.
\]
Therefore the relationship between $c(\varphi)_{i\bar{j}}$ and $\Psi$ will be
\begin{equation}\label{e23}
-\Psi=\sum ic(\varphi)_{i\bar{j}}dz_{i}\wedge d\bar{z}_{j}.
\end{equation}

The following result is a quick consequence of the formula (\ref{e22}).
\begin{proposition}[Kobayashi, \cite{Kob75,Kob96}]\label{p24}
$(\mathcal{O}_{E}(1),\varphi)$ is ample (hence $E$ is ample) if and only if $(E^{\ast},G)$ is strictly negative in the sense of Kobayashi. In particular, if $G$ is a Hermitian metric on $E^{\ast}$, it is also equivalent to say that $(E,G^{\ast})$ is strictly positive in the sense of Griffiths.
\end{proposition}

\subsection{Q-twisted bundle}
We recall the $\mathbb{Q}$-twisted bundle from \cite{Laz04} in the end of this section.
\begin{definition}[$\mathbb{Q}$-twisted bundles]\label{d26}
A $\mathbb{Q}$-twisted vector bundle
\[
E<\delta>
\]
over $Y$ is an ordered pair consisting of a vector bundle $E$ over $Y$, defined up to isomorphism, and a $\mathbb{Q}$-divisor $\delta$ on $Y$.
\end{definition}

The definition of positivity for a vector bundle extends to a $\mathbb{Q}$-twisted bundle. Before introducing this, we take the chance to recall the positivity concerning a $\mathbb{Q}$-divisor in \cite{Dem12} first. Certainly this definition also works for a line bundle.
\begin{definition}\label{d27}
Let $\delta$ be a $\mathbb{Q}$-divisor on $Y$.
\begin{enumerate}
\item The K\"{a}hler cone is the set $\mathcal{K}\subset H^{1,1}(Y,\mathbb{R})$ of cohomology classes $[\omega]$ of K\"{a}hler forms. $\delta$ is ample if $\delta\in\mathcal{K}$.
  \item $\delta$ is nef if $\delta\in\overline{\mathcal{K}}$. Here $\overline{\mathcal{K}}$ is the closure of $\mathcal{K}$.
  \item The pseudo-effective cone is the set $\mathcal{E}\subset H^{1,1}(Y,\mathbb{R})$ of cohomology classes [T] of
closed positive currents of type $(1,1)$. $\delta$ is pseudo-effective if $\delta\in\mathcal{E}$.
   \item  $\delta$ is big if $\delta\in\mathcal{E}^{o}$. Here $\mathcal{E}^{o}$ is the interior of $\mathcal{E}$.
\end{enumerate}
\end{definition}
The positivity of a $\mathbb{Q}$-twisted bundle is defined as follows:
\begin{definition}\label{d28}
Let $E<\delta>$ be a $\mathbb{Q}$-twisted bundle.
\begin{enumerate}
\item $E<\delta>$ is ample, if $\mathcal{O}_{E}(1)+\pi^{\ast}\delta$ is an ample $\mathbb{Q}$-divisor on $X$.
  \item $E<\delta>$ is nef, if $\mathcal{O}_{E}(1)+\pi^{\ast}\delta$ is a nef $\mathbb{Q}$-divisor on $X$.
  \item $E<\delta>$ is big, if $\mathcal{O}_{E}(1)+\pi^{\ast}\delta$ is a big $\mathbb{Q}$-divisor on $X$.
\end{enumerate}
\end{definition}

\section{The singular Hermitian metric}
\subsection{The definition}
Using the same notations as Sect.2.1, we have the following definition. We emphasise that throughout this paper, we only consider those locally $L^{1}$-bounded functions on $X$ whose singular part doesn't dominate $Y$.
\begin{definition}\label{d31}
Fix a smooth Hermtian metric $h_{0}$ on $\mathcal{O}_{E}(1)$ as the reference metric. For any $\varphi\in L^{1}_{\textrm{loc}}$, let
\[
Y_{\varphi}:=\{z\in Y;\varphi|_{X_{z}}\textrm{ is well-defined}\}.
\]
Then we define
\[
\begin{split}
\mathcal{H}(X):=&\{\varphi\in L^{1}_{\textrm{loc}}(X);\textrm{when }z\in Y_{\varphi}, \varphi|_{X_{z}}\textrm{ is smooth and}\\
 &(i\Theta_{\mathcal{O}_{X}(1),h_{0}}+i\partial\bar{\partial}\varphi)|_{X_{z}}\textrm{ is strictly positive}\}.
\end{split}
\]
\end{definition}
Notice that for any $\varphi\in\mathcal{H}(X)$ and $z\in Y_{\varphi}$,
\[
\omega_{\varphi,z}:=(i\Theta_{\mathcal{O}_{X}(1),h_{0}}+i\partial\bar{\partial}\varphi)|_{X_{z}}
\]
is actually a K\"{a}her metric on $X_{z}$. Based on this observation we define the singular Hermitian metric on $E$ as follows.
\begin{definition}=[Definition \ref{d11}]\label{d32}
The singular Hermitian metric $H$ on $E$ is a map with the form that
\[
H(z,Z)=\int_{X_{z}}|Z|^{2}_{h_{0}}e^{-\varphi}\frac{\omega^{r}_{\varphi,z}}{r!}
\]
where $\varphi\in\mathcal{H}(X)$. Here we use the fact that $\pi_{\ast}\mathcal{O}_{E}(1)\simeq E$.
\end{definition}
Note that when $\varphi$ is smooth, it is just the traditional $L^{2}$-metric induced by $\varphi$. When $\varphi$ is merely $L^{1}$-bounded, it is easy to see that $H(z,Z)$ is an a.e. finite (but may not be $L^{1}$-bounded) function on $Y$.

\subsection{Comparison with the former definition}
Recall that \cite{Cat98} and \cite{Rau15} provide two types of singular Hermitian metric. We will call them Cataldo's singular metric and Raufi's singular metric respectively. In \cite{Rau15}, there is a detailed comparison between these two definitions.

Now we should prove Theorem \ref{t1}, which says that our definition is compatible with Raufi's singular metric.
\begin{proof}[Proof of Theorem \ref{t1}]
(1) Keep the notations. It is enough to prove that any smooth Hermitian metric $H$ on $E$ can be rewritten as an $L^{2}$-metric. Indeed, take $\varphi$ to be the metric on $\mathcal{O}_{E}(1)$ corresponding to $H$, we claim that
\[
H(z,Z)=\int_{X_{z}}|Z|^{2}e^{-\varphi}\frac{\omega^{r}_{\varphi,z}}{r!}
\]
up to a constant, hence finish the proof of (1).

Now we prove the claim above. In fact, it is already done in \cite{LSY06} by comparing the curvature associated with both metrics. Here we reprove it by locally representing the $L^{2}$-metric. For a coordinate ball $U\subset Y$, we take the local trivialization of $E$ as
\[
    \begin{split}
        E|_{U}\simeq U\times\mathbb{C}^{r+1}&\rightarrow U\\
        (z,Z=(Z_{0},...,Z_{r}))&\mapsto z.
    \end{split}
\]
Accordingly, the local coordinate of the fibration
\[
    \pi:X\rightarrow Y
\]
is
\[
    \begin{split}
        X|_{f^{-1}(U)}\simeq U\times\mathbb{P}^{r}&\rightarrow U\\
        (z,[Z^{\ast}])&\mapsto z,
    \end{split}
\]
where $(Z^{\ast})$ means the dual coordinate. Let $W_{A}=\{Z^{\ast}_{A}\neq0\}$ with $0\leqslant A\leqslant r$, then we have
\[
    \begin{split}
        X|_{U\times W_{A}}&\rightarrow U\\
        (z,w_{A})&\mapsto z
    \end{split}
\]
with $w_{A}=(w^{0}_{A},...,1^{A},...,w^{r}_{A})=(\frac{Z^{\ast}_{0}}{Z^{\ast}_{A}},...,\frac{Z^{\ast}_{r}}{Z^{\ast}_{A}})$. Here $1^{A}$ means the $A$-th component equals $1$. In this setting the local trivialization of $\mathcal{O}_{E}(1)$ is
\[
    \mathcal{O}_{E}(1)|_{U\times W_{A}}\simeq U\times W_{A}\times\mathbb{C}
\]
with coordinate $(z,w_{A},\xi)$. Recall there exits an isomorphism
\[
\pi_{\ast}(\mathcal{O}_{E}(1))=E,
\]
hence a section of $E$ maps to a section of $\pi_{\ast}(\mathcal{O}_{E}(1))$. This map can be explicitly written down as follows: fix a holomorphic basis $\{u_{0},...,u_{r}\}$ of $E$, then an arbitrary section
\[
    \begin{split}
        u: U&\rightarrow E|_{U}\\
        z&\mapsto (z,(Z_{0}(z),...,Z_{r}(z))
    \end{split}
\]
maps to
\[
    \begin{split}
        \xi: U\times W_{A}&\rightarrow\mathcal{O}_{E}(1)|_{U\times W_{A}}\\
        (z,w_{A})&\mapsto(z,w_{A},\xi_{A})
    \end{split}
\]
with
\[
    \xi_{A}(Z_{0},...,Z_{r})=Z_{A}+\sum_{0\leqslant\alpha\leqslant r,\alpha\neq A}Z_{\alpha}w^{\alpha}_{A}.
\]

Now the $L^{2}$-metric is defined as
\begin{equation}\label{e31}
    G(z,Z):=\int_{X_{z}}|\xi_{A}|^{2}e^{-\varphi}\frac{\omega^{r}_{\varphi,z}}{r!}.
\end{equation}
In this context, if locally we expand $G(z,Z)=\sum G_{\alpha\bar{\beta}}Z_{\alpha}\bar{Z}_{\beta}$, $(G_{\alpha\bar{\beta}})$ will be a matrix with the form as follows:

1. $\alpha=\beta=A$, then
\[
    G_{\alpha\bar{\beta}}=\int_{X_{y}}e^{-\varphi(z,w_{A})}\frac{\omega^{r}_{\varphi,z}} {r!};
\]

2. $\alpha=A,\beta\neq A$, then
\[
    G_{\alpha\bar{\beta}}=\int_{X_{z}}\bar{w}^{\beta} _{A}e^{-\varphi(z,w_{A})}\frac{\omega^{r}_{\varphi,z}}{r!};
\]

3. $\alpha\neq A,\beta=A$, then
\[
    G_{\alpha\bar{\beta}}=\int_{X_{z}}w^{\alpha}_{A}e^ {-\varphi(z,w_{A})}\frac{\omega^{r}_{\varphi,z}}{r!};
\]

4. $\alpha\neq A, \beta\neq A$, then
\[
    G_{\alpha\bar{\beta}}= \int_{X_{z}}w^{\alpha}_{A}\bar{w}^{\beta}_{A}e^{-\varphi(z,w_{A})}\frac{\omega^{r}_{\varphi,z}}{r!}.
\]
Now we take the holomorphic basis $\{u_{0},...,u_{r}\}$ of $E$ such that it is an orthonormal basis with respect to $H$ (as well as $H^{\ast}$) at a fixed point $z_{0}\in Y$. In other words, $H(z_{0})$ and $H^{\ast}(z_{0})$ are identity matrices. Accordingly, $\varphi(z)$, which is induced by $H(z)$, is the standard Fubini--Study metric on $X_{z_{0}}=\mathbb{P}^{r}$. Then at this point
\[
\begin{split}
G(z_{0},Z)&=\sum G_{\alpha\bar{\beta}}Z_{\alpha}\bar{Z}_{\beta}\\
      &=\frac{\textrm{Vol}(\mathbb{P}^{r})}{|Z^{\ast}_{0}|^{2}+\cdots+|Z^{\ast}_{r}|  ^{2}}\\
      &=\textrm{Vol}(\mathbb{P}^{r})(|Z_{0}|^{2}+\cdots+|Z_{r}|^{2})\\
      &=\textrm{Vol}(\mathbb{P}^{r})H(z_{0},Z).
\end{split}
\]
The second equality is an elementary computation based on the local representative of $G_{\alpha\bar{\beta}}$ above. We eventually obtain that
\[
G(z,Z)=\textrm{Vol}(\mathbb{P}^{r})H(z,Z)
\]
for every $z\in Y$, and the proof is complete.

(2) When $E$ is a line bundle, $X=\mathbb{P}(E^{\ast})=Y$ and $\mathcal{O}_{E}(1)=E$. Hence $\mathcal{H}(X)=L^{1}_{\textrm{loc}}(Y)$. Let $\varphi\in L^{1}_{\textrm{loc}}(Y)$, then $H$ in Definition \ref{d32} degenerates as
\[
H(z,Z)=|Z|^{2}e^{-\varphi},
\]
which is exactly the singular metric on a line bundle.

(3) Recall that Raufi's singular metric in \cite{Rau15} refers to a measurable map $H$ from $Y$ to the space of non-negative Hermitian forms on the fibres of $E$. So if moreover $H\in M^{o}_{R}(E)$, $H(z)$ must be a positive-definite Hermitian matrix as long as it is finite at $z\in Y$. Let $\varphi$ be its corresponding (singular) metric on $\mathcal{O}_{E}(1)$. It is easy to verify that $\varphi\in\mathcal{H}(X)$. Now the computation in (1) shows that
\[
H=\int_{X_{z}}|Z|^{2}e^{-\varphi}\frac{\omega^{r}_{\varphi,z}}{r!}
\]
up to a constant, namely $H\in M^{o}_{W}(E)$. Notice this computation is done on each fibre $X_{z}$ with $z\in Y_{\varphi}$. In particular, it does not require that $H$ smoothly depends on $z$. In summary we have proved that $M^{o}_{R}(E)\subset M^{o}_{W}(E)$.

In order to prove $M^{o}_{W}(E)\subset M^{o}_{R}(E)$, it is enough to show that
\[
\int_{X_{z}}|Z|^{2}e^{-\varphi}\frac{\omega^{r}_{\varphi,z}}{r!}
\]
is a non-negative Hermitian form at each $z\in Y_{\varphi}$. It is obvious concerning the local representation in (1).
\end{proof}

In the rest part, by saying a singular Hermitian metric $H$ on $E$, we always refer to an $H\in M_{W}(E)$. Comparing with Raufi's singular metric, $H(z,\cdot)$ must be a positive-definite matrix when it is finite, whereas $H(z,Z)$ is not necessary to be measurable. There are two points showing the nature of this adjustment.

Remember that in \cite{Rau15}, there is an example (Theorem 1.5) showing that in general the associated curvature current cannot be defined for a Raufi's singular metric $h$. Indeed, \cite{Rau15} defines a singular metric $h$ on a trivial vector bundle $\Delta\times\mathbb{C}^{2}$ over the unit disc $\Delta$. The main reason that $\Theta_{h}:=\bar{\partial}(h^{-1}\partial h)$ doesn't have measure coefficients is that $h$ degenerates (as a matrix) at the origin. So it is reasonable to excluded this situation in our definition.

The second observation comes from the line bundle case. Recall that a singular metric on a line bundle $L$ \cite{Dem12} is a function on its total space with the form that $|\xi|^{2}e^{-\varphi}$, where $\xi$ is a section of $L$ and $\varphi$ is a measurable function on $Y$. There are various examples showing that $e^{-\varphi}$ is not need to be integrable, hence neither is $|\xi|^{2}e^{-\varphi}$. This observation inspires us to consider the possibility to treat the non-measurable map as a singular metric on a vector bundle $E$ of higher rank.

\section{The Finsler geometry in singular case}
This section is devoted to prove Theorem \ref{t2}. Note that in the following part, $Y$ is assumed to be compact. To begin with, we give the definition of Griffiths positivity for a singular Hermitian vector bundle $(E,H)$. Notice that $H$ naturally induces a metric $H^{\ast}$ on $E^{\ast}$, and hence a singular metric $\varphi$ on $\mathcal{O}_{E}(1)$. Keep notations in Sect.2.1, we have
\begin{definition}[=Definition \ref{d12}]\label{d41}
Let $\omega$ be a Hermitian metric on $X=\mathbb{P}(E^{\ast})$. We say $(E,H)$ is positive in the sense of Griffiths, if
\[
i\Theta_{\mathcal{O}_{E}(1),\varphi}\geqslant0
\]
in the sense of current. $(E,H)$ is strictly positive in the sense of Griffiths, if
\[
i\Theta_{\mathcal{O}_{E}(1),\varphi}\geqslant\omega.
\]
\end{definition}

Next we present Theorem \ref{t2}, which gives an equivalent description for a Griffiths positive vector bundle $(E,H)$. Recall that
\[
\mathcal{H}_{i}(X):=\{\varphi\in\mathcal{H}(X);\varphi\textrm{ is induced by an }H\in M_{W}(E)\}.
\]
\begin{theorem}[=Theorem \ref{t2}]
Consider the $\varphi\in\mathcal{H}(X)$ such that
\[
i\Theta_{\mathcal{O}_{E}(1),\varphi}\geqslant0.
\]
Let $H_{\varphi}$ be the singular Hermitian metric defined by $\varphi$. The following two statements are equivalent.
\begin{enumerate}
  \item[(a)] $\varphi\in\mathcal{H}_{i}(X)$. More precisely, $\varphi$ is induced by $H_{\varphi}$;
  \item[(b)] $\varphi$ induces an isometry between the canonical isomorphism
\[
  K_{X/Y}\simeq\mathcal{O}_{X}(-r-1)+\pi^{\ast}\det E.
\]
\end{enumerate}
\end{theorem}
The statement (a) infers that $(E,H_{\varphi})$ is positive in the sense of Griffiths.

Before proceeding to the proof, we should explain the isometry in this theorem first. Recall that for any $\varphi\in\mathcal{H}(X)$, $\omega_{\varphi,z}$ is a K\"{a}hler metric on $X_{z}$ if $z\in Y_{\varphi}$. Therefore it induces a smooth metric $\omega^{r}_{\varphi,z}$ on $K_{X_{z}}$. On the other hand, $H_{\varphi}$ induces a singular metric $\det H_{\varphi}$ on $\det E$. Now we say that $\varphi$ induces an isometry between the canonical isomorphism
\[
-K_{X/Y}\simeq\mathcal{O}_{E}(r+1)+\pi^{\ast}\det E^{\ast}
\]
if
\begin{equation}\label{e41}
\omega^{r}_{\varphi,z}=Ce^{-(r+1)\varphi}\cdot\pi^{\ast}\det H^{\ast}_{\varphi}dw\wedge d\bar{w}
\end{equation}
on $X_{z}$ when $z\in Y_{\varphi}$. In the following part, our discussion will always focus on $Y_{\varphi}$.

The constant $C$ is determined as follows.
\begin{lemma}\label{l41}
The normalization constant $C$ in (\ref{e41}) equals to $\frac{1}{((r+1)!)^{r}}$.
\begin{proof}
The proof is a continuation of the computation in the proof of Theorem \ref{t1}. Using the same notation there, if locally
\[
H_{\varphi}=\sum H_{\alpha\bar{\beta}}Z_{\alpha}\bar{Z}_{\beta},
\]
the determination $\det H_{\varphi}=\det H_{\alpha\bar{\beta}}Z^{\ast}\wedge\bar{Z}^{\ast}$ defines a metric on $\det E$. Here
\[
Z^{\ast}\wedge\bar{Z}^{\ast}:=Z^{\ast}_{1}\wedge\cdots\wedge Z^{\ast}_{r}\wedge\bar{Z}^{\ast}_{1}\wedge\cdots\wedge\bar{Z}^{\ast}_{r}.
\]
Then the equation (\ref{e41}) can be reformulated as
\begin{equation}\label{e42}
    (i\partial\bar{\partial}_{w_{A}}\varphi)^{r}=C\cdot C_{z}e^{-(r+1)\varphi(z,w_{A})}dw_{A}\wedge d\bar{w}_{A},
\end{equation}
with $C_{z}=\pi^{\ast}\det H^{-1}_{i\bar{j}}(z)$. It is now a standard K\"{a}hler--Einstein equation on $X_{z}=\mathbb{P}^{r}$. Thus the solution of the equation (\ref{e42}) must have the following form:
\[
    \varphi=\log(1+\sum_{A}|w_{A}|^{2})-\frac{1}{r+1}\log(C_{z}+C).
\]
Replacing $\varphi$ by $\log(1+\sum_{A}|w_{A}|^{2})-\frac{1}{r+1}\log(C_{z}+C)$ in the local expression of $H_{\alpha\bar{\beta}}$, we get
\[
    \begin{split}
        C_{z}&=\pi^{\ast}\det H^{-1}_{\alpha\bar{\beta}}(z)\\
             &=C_{z}\cdot C\det(\int_{\mathbb{P}^{r}}\frac{Z^{\ast}_{\alpha}Z^{\ast}_{\bar{\beta}}}{|Z^{\ast}|^{2}}\frac{\omega^{r}_{FS}}{r!})^{-1}\\
             &=((r+1)!)^{r}C_{z}\cdot C.
    \end{split}
\]
It means that $C=\frac{1}{((r+1)!)^{r}}$.
\end{proof}
\end{lemma}

Now we are ready to prove Theorem \ref{t2}.
\begin{proof}[Proof of Theorem \ref{t2}]
One direction is easy. If $\varphi$ is induced by a singular Hermitian metric, say $H$, on $E$, then we can take a holomorphic basis of $E$ such that $H$ is the identity matrix at the given point $z_{0}\in Y_{\varphi}$. Accordingly, $\varphi|_{X_{z_{0}}}$ is the Fubini--Study metric on $X_{z_{0}}=\mathbb{P}^{r}$, hence solves the equation (\ref{e42}). It is equivalent to say that $\varphi$ induces an isometry between
\[
K_{X/Y}\simeq\mathcal{O}_{E}(-r-1)+\pi^{\ast}\det E.
\]

Next we shall prove the opposite direction. Let $\psi$ be the metric on $\mathcal{O}_{E}(1)$ induced by $H_{\varphi}$. In order to prove $(b)\Rightarrow(a)$, we should verify that $\varphi=\psi$. Observe that $\varphi$ is a quasi-plurisubharmonic function on $X$, it is enough to prove that $i\partial\bar{\partial}\varphi=i\partial\bar{\partial}\psi$.  It is a local property, so we could assume that $\varphi$ is smooth without loss of generality by standard approximation technique. At this time, the curvature $\Theta^{H_{\varphi}}$ associated with $H_{\varphi}$ is well-defined. Using the same notations as in Sect.2, it is then left to prove that
\begin{equation}\label{e43}
   \sum i\Theta^{H_{\varphi}}_{\alpha\bar{\beta}i\bar{j}}w_{\alpha}\bar{w}_{\beta}dz_{i}\wedge d\bar{z}_{j}=-\Psi^{T}G
\end{equation}
up to a constant. Here $G$ is the Finsler metric on $E^{\ast}$ corresponding to $\varphi$, $\Psi$ is the Kobayashi curvature associated with $\varphi$, and $\Psi^{T}$ means to take the conjugate transpose with respect to indexes $\alpha,\beta$. In other words, if $\Psi=\sum i\textrm{K}^{\varphi}_{\alpha\bar{\beta}i\bar{j}}\frac{w_{\alpha}\bar{w}_{\beta}}{G}dz_{i}\wedge d\bar{z}_{j}$, then
\[
    \Psi^{T}=\sum i\textrm{K}^{\varphi}_{\beta\bar{\alpha}i\bar{j}}\frac{w_{\alpha}\bar{w}_{\beta}}{G}dz_{i}\wedge d\bar{z}_{j}.
\]
Indeed, if the formula (\ref{e43}) holds, we must have $i\partial\bar{\partial}\psi=i\partial\bar{\partial}\varphi$ by Proposition \ref{p23}.

The formula (\ref{e43}) is implicitly included in the main theorem of \cite{Nau17}, and we provide a simple explanation here for readers' convenience. It is divided into three parts.

1. By \cite{Ber09}, we know that for any section
\[
u^{1}=\sum Z^{1}_{\alpha}u_{\alpha}\textrm{ and }u^{2}=\sum Z^{2}_{\alpha}u_{\alpha}\in E
\]
and complex vector $v=(v_{1},...,v_{n})$, the curvature formula concerning the $L^{2}$-metric $H_{\varphi}$ is as follows:
\begin{equation}\label{e44}
\begin{split}
    \sum_{i,j}(i\Theta^{H_{\varphi}}_{i\bar{j}}u^{1},u^{2})_{H_{\varphi}}v_{i}\bar{v}_{j}=&\sum_{i,j}\int_{X_{z}}ic(\psi)_{i\bar{j}}\xi^{1}\bar{\xi}^{2} v_{i}\bar{v}_{j}e^{-\varphi}\frac{\omega^{r}_{\varphi,z}}{r!}\\
&+\int_{X_{z}}\eta^{1}\wedge\bar{\eta}^{2}dV_{z}.
\end{split}
\end{equation}
Here $\psi$ is given by $e^{-\psi}dw\wedge d\bar{w}=e^{-\varphi}\omega^{r}_{\varphi,z}$, and $\xi^{1},\xi^{2}$ are given by $u^{1},u^{2}$. More precisely, the datum $(Z^{1}_{\alpha})$ and $(Z^{2}_{\alpha})$ will uniquely determine the sections $\xi^{1},\xi^{2}$ of $\mathcal{O}_{E}(1)$ as is shown in the proof of Theorem \ref{t1}. On the other hand, $\eta^{1},\eta^{2}$ are given by
\[
    \eta^{1}=\frac{\bar{\partial}\xi^{1}\wedge\sum v_{i}\hat{dz_{i}}}{dz}
\]
and
\[
\eta^{2}=\frac{\bar{\partial}\xi^{2}\wedge\sum v_{i}\hat{dz_{i}}}{dz}
\]
respectively. Here $\hat{dz_{i}}$ denotes the wedge product of all $dz_{j}$ except $dz_{i}$. In other words,
\[
\eta^{i}=\kappa^{\varphi}_{z}\rfloor\xi^{i}\textrm{ with }i=1,2,
\]
where $\kappa^{\varphi}_{z}$ is the Kodaira--Spencer class.

2. When $\varphi$ induces the isometry (\ref{e11}), it is proved in \cite{Nau17} that (Proposition 1) $\kappa^{\varphi}_{z}=0$, and (Corollary 1)
\[
    \pi^{\ast}R^{\det E}_{H,i\bar{j}}=\int_{X_{z}}c(\varphi)_{i\bar{j}}\omega^{r}_{\varphi,z}.
\]
Strictly speaking, the context here is slightly different from \cite{Nau17}. However, one could see that $\varphi$ is not need to be strictly plurisubharmonic as in \cite{Nau17} after carefully checking the calculation there. Now $\eta^{1}=\eta^{2}=0$ since $\kappa^{\varphi}_{z}=0$. Moreover,  $\varphi$ is a K\"{a}hler--Einstein metric on $X_{z}=\mathbb{P}^{r}$, hence gives a polarization \cite{GrH78}. Combine with the formula (\ref{e23}), we obtain that
\[
\int_{X_{z}}c(\varphi)_{i\bar{j}}\omega^{r}_{\varphi,z}=-\int_{\mathbb{P}^{r}}\textrm{K}^{\varphi}_{\sigma\bar{\tau}i\bar {j}}\frac{w_{\sigma}\bar{w}_{\tau}}{|w|^{2}}\omega^{r}_{FS}.
\]
Then apply a clever eigenvalue argument presented in \cite{Nau17}, we can arrange the thing that $\textrm{K}^{\varphi}_{\sigma\bar{\tau}i\bar {j}}$ is independent of $w=(w_{\sigma})$, hence we eventually have
\[
\int_{X_{z}}c(\varphi)_{i\bar{j}}\omega^{r}_{\varphi,z}=-\textrm{K}^{\varphi}_{\sigma\bar{\tau}i\bar {j}}\int_{\mathbb{P}^{r}}\frac{w_{\sigma}\bar{w}_{\tau}}{|w|^{2}}\omega^{r}_{FS}.
\]

3. Now we are ready to calculate the curvature at a given point $z_{0}\in Y$. Take $u^{1}=u_{\alpha}$ and $u^{2}=u_{\beta}$ (hence $\xi^{1}=w^{\alpha},\xi^{2}=w^{\beta}$ at this time), and remember that we can always assume that $\{u_{\alpha}\}$ is orthonormal with respect to $H_{\varphi}$ at $z_{0}$, the left hand side of formula (\ref{e44}) becomes $i\Theta^{H_{\varphi}}_{\alpha\bar{\beta}i\bar{j}}v_{i}\bar{v}_{j}$. On the other hand, as $\varphi$ polarizes $X_{z}$, the right hand side of formula (\ref{e44}) actually equals to
\[
\begin{split}
    =&iv_{i}\bar{v}_{j}\int_{X_{z}}c(\psi)_{i\bar{j}}w_{\alpha}\bar{w}_{\beta}e^{-\varphi}\frac{\omega^{r}_{\varphi,z}}{r!}\\
    =&iv_{i}\bar{v}_{j}\int_{X_{z}}c((r+2)\varphi-\pi^{\ast}\log\det H)_{i\bar{j}}w_{\alpha}\bar{w}_{\beta}e^{-\varphi}\frac{\omega^{r}_{\varphi,z}}{r!}\\
    =&-r!iv_{i}\bar{v}_{j}\int_{\mathbb{P}^{r}}((r+2)\textrm{K}^{\varphi}_{\sigma\bar{\tau}i\bar{j}}\frac{w_{\sigma}\bar{w}_{\tau}}{|w|^{2}}+\pi^{\ast} R^{\det E}_{H,i\bar{j}})\frac{w_{\alpha}\bar{w}_{\beta}}{|w|^{2}}\frac{\omega^{r}_{FS}}{r!}\\
    =&-r!iv_{i}\bar{v}_{j}((r+2)\textrm{K}^{\varphi}_{\sigma\bar{\tau}i\bar{j}}\int_{\mathbb{P}^{r}}\frac{w_{\alpha}\bar{w}_{\beta}w_{\sigma}\bar{w}_ {\tau}}{|w|^{4}}\frac{\omega^{r}_{FS}}{r!}\\
    &-\delta_{\alpha\bar{\beta}}K^{\varphi}_{\sigma\bar{\tau}i\bar{j}}\int_{\mathbb{P}^{r}}\frac{w_{\sigma} \bar{w}_{\tau}}{|w|^{2}}\frac{\omega^{r}_{FS}}{r!})\\
    =&-r!iv_{i}\bar{v}_{j}\frac{\textrm{K}^{\varphi}_{\sigma\bar{\tau}i\bar{j}}}{(r+1)!}((\delta_{\alpha\bar{\beta}}\delta_{\sigma\bar{\tau}}+\delta_ {\alpha\bar{\tau}}\delta_{\sigma\bar{\beta}})-\delta_{\alpha\bar{\beta}}\delta_{\sigma\bar{\tau}})\\
    =&-iv_{i}\bar{v}_{j}\frac{\textrm{K}^{\varphi}_{\beta\bar{\alpha}i\bar{j}}}{r+1},
\end{split}
\]
By now, we have successfully proved the formula (\ref{e43}). As a result, $\varphi$ and $\psi$ are equal up to a constant. The proof is complete.
\end{proof}

Now let $\mathcal{H}_{h_{0}}(X)$ be the set of all the functions $\varphi\in\mathcal{H}(X)$ such that
\begin{enumerate}
  \item[(i)] $i\Theta_{\mathcal{O}_{E}(1),h_{0}}+i\partial\bar{\partial}\varphi\geqslant0$ in the sense of current;
  \item[(ii)] $\varphi$ induces an isometry between
\[
K_{X/Y}\simeq\mathcal{O}_{X}(-r-1)+\pi^{\ast}\det E.
\]
\end{enumerate}

Then we summarize the description for a Griffiths positive vector bundle as follows.
\begin{corollary}[=Corollary \ref{c11}]
Let $E$ be a vector bundle. Then there is a one-one correspondence between the set of the (singular) Hermitian metrics $H$ on $E$ such that $(E,H)$ is Griffiths positive in the sense of Definition \ref{d12} and $\mathcal{H}_{h_{0}}(X)$.
\begin{proof}
By Theorem \ref{t2}.
\end{proof}
\end{corollary}

\section{The positivity and the vanishing theorem}
\subsection{Setup}
Let $H$ be a (singular) Hermitian metric on $E$. We will  define the Griffiths and Nakano positivities in Definition \ref{d51}. After that, in order to distinguish them with the notions appear in \cite{Rau15}, we say that $(E,H)$ is positive in the sense of Griffiths if it meets the requirement in Definition \ref{d51}. We say that $(E,H)$ is Griffiths or Nakano positive in the sense of Raufi, if it satisfies the properties in Definition 1.2 and 1.8 of \cite{Rau15} accordingly.
\subsection{The Griffiths and Nakano positivities}
In this section, we discuss the positivity associated with the singular Hermitian metric. Keep notations as in Sect.2.1, we have
\begin{definition}\label{d51}
Let $H$ be a (singular) Hermitian metric on $E$, and let $\varphi$ be the corresponding metric on $\mathcal{O}_{E}(1)$. Let $q$ be a rational number. Then
\begin{enumerate}
  \item $(E,H)$ is (strictly) positive in the sense of Griffiths, if $i\partial\bar{\partial}\varphi$ is (strictly) positive on $X$.
  \item $(E<q\det E>,H)$ is (strictly) positive in the sense of Griffiths, if $i\partial\bar{\partial}\varphi+q\pi^{\ast}c_{1}(\det E,\det H)$ is (strictly) positive on $X$.
  \item $(E,H)$ is (strictly) strongly positive in the sense of Nakano, if the $\mathbb{Q}$-twisted vector bundle
\[
(E<-\frac{1}{r+2}\det E>,H)
\]
is (strictly) positive in the sense of Griffiths.
  \item $(E<q\det E>,H)$ is (strictly) strongly positive in the sense of Nakano, if the $\mathbb{Q}$-twisted vector bundle
\[
(E<\frac{2q-1}{r+2}\det E>,H_{\varphi})
\]
is (strictly) positive in the sense of Griffiths.
\end{enumerate}
\end{definition}

We will see in Theorem \ref{t3} that the Griffiths positivity defined here coincides with the one in Raufi's sense, while the strongly Nakano positivity is slightly stronger than the Nakano positivity in Raufi's sense. The advantage to discuss the strongly Nakano positivity is that it is not necessary to require the vector bundle to be positive in the sense of Griffiths a priori.

Now we shall prove Theorem \ref{t3}.
\begin{proof}[Proof of Theorem \ref{t3}]
(1) Since $H$ is a smooth Hermitian metric, the corresponding metric $\varphi$ on $\mathcal{O}_{E}(1)$ is also smooth. By Proposition \ref{p24}, $(E,H)$ is (strictly) positive in the sense of Griffiths if and only if $(\mathcal{O}_{E}(1),\varphi)$ is (ample) semi-positive. Then we finish the proof for the Griffiths positivity part.

Now assume that $(E,H)$ is (strictly) strongly Nakano positive, i.e.
\[
(E<-\frac{1}{r+2}\det E>,H)
\]
is (strictly) Griffiths positive as a $\mathbb{Q}$-twisted bundle. By definition, it is equivalent to say that $\mathcal{O}_{E}(r+2)-\pi^{\ast}\det E$ is an (ample) semi-positive line bundle on $X$. Then we apply the curvature formula (\ref{e44}) to this line bundle to obtain that
\[
\pi_{\ast}(K_{X/Y}+\mathcal{O}_{E}(r+2)-\pi^{\ast}\det E)=E
\]
is (strictly) positive in the sense of Nakano.

(2) When $(E,H)$ is positive in the sense of Griffiths, $\varphi=\log H^{\ast}$ is plurisubharmonic by definition. It exactly means that $(E^{\ast},H^{\ast})$ is Griffiths negative in the sense of Raufi. Equivalently $(E,H)$ is Griffiths positive in the sense of Raufi.

When $(E,H)$ is Griffiths positive in the sense of Raufi, locally there exists an increasing regularising sequence $\{H_{\nu}\}$ such that $(E,H_{\nu})$ is positive in the sense of Griffiths for every $\nu$ by Proposition 1.3 in \cite{Rau15}. Let $\varphi_{\nu}$ be the corresponding metric on $\mathcal{O}_{E}(1)$, then $\varphi_{\nu}$ is convergent to $\varphi$. Moreover, apply Proposition \ref{p23} to $H_{\nu}$, we get that
\[
i\partial\bar{\partial}\varphi_{\nu}=-\Psi_{\nu}+\omega_{FS,\nu}.
\]
Since $(E,H_{\nu})$ is positive in the sense of Griffiths, $\Psi_{\nu}$ is negative by Proposition \ref{p24}. It means that $\varphi_{\nu}$ is plurisubharmonic, so will be $\varphi$. Hence $(E,H)$ is positive in the sense of Griffiths.

(3) Assume that $(E,H)$ is (strictly) positive in the sense of Griffiths. By (2) locally there exists an increasing regularising sequence $\{H_{\nu}\}$ such that $(E,H_{\nu})$ is also positive in the sense of Griffiths. Since the desired conclusion is local, we may without loss of generality assume that $H$ is smooth itself. Now we applied the formula
\begin{equation}\label{e51}
\pi^{\ast}R^{\det E}_{\det H,i\bar{j}}=\int_{X_{z}}c(\varphi)_{i\bar{j}}\omega^{r}_{\varphi}
\end{equation}
appeared before. As a consequence, $R^{\det E}_{\det H,i\bar{j}}$ is (strictly) positive on $Y$ provided that $i\partial\bar{\partial}\varphi(>)\geqslant0$ on $X$. It exactly means that $(\det E,\det H)$ is (big) pseudo-effective.

(4) Let $\varphi$ be the corresponding metric on $\mathcal{O}_{E}(1)$. By definition,
\[
i\partial\bar{\partial}\varphi-\frac{1}{r+2}\pi^{\ast}c_{1}(\det E,\det H)
\]
is positive. Equivalently,
\[
i(r+2)\partial\bar{\partial}\varphi-\pi^{\ast}c_{1}(\det E,\det H)
\]
is positive. The aim is to prove that $c_{1}(\det E,\det H)$ is positive, hence $i\partial\bar{\partial}\varphi$ is also positive. But it is not a straightforward application of (3). The proof is more involved.

Indeed, $H$ induces a metric $S^{r+2}H$ on the $(r+2)$-th symmetric product $S^{r+2}E$, hence a metric $\psi$ on $\mathcal{O}_{S^{r+2}E}(1)$. Now let
\[
i:\mathbb{P}(E^{\ast})\hookrightarrow\mathbb{P}(S^{r+2}E^{\ast})
\]
be the Vernoese embedding \cite{GrH78}. Then we have
\[
i^{\ast}c_{1}(\mathcal{O}_{S^{r+2}E}(1),\psi)=c_{1}(\mathcal{O}_{E}(r+2),(r+2)\varphi).
\]
As a result, $i\partial\bar{\partial}_{h}\psi$ has the same
positivity as $i(r+2)\partial\bar{\partial}_{h}\varphi$, where $\partial\bar{\partial}_{h}$ means to take the derivative in the horizontal direction. Observe that $\psi$ is induced by a Hermitian metric on $S^{r+2}E$, $i\partial\bar{\partial}_{v}\psi$ must be positive, where $\partial\bar{\partial}_{v}$ means to take the derivative in the vertical direction. In summary, let $\Pi:\mathbb{P}(S^{r+2}E^{\ast})\rightarrow Y$ be the projection, we have
\[
i\partial\bar{\partial}\psi-\Pi^{\ast}c_{1}(\det E,\det H)\geqslant0.
\]
By definition we conclude that $(S^{r+2}E\otimes\det E^{\ast},S^{r+2}H\otimes\det H^{\ast})$ is positive in the sense of Griffiths. Using the same regularising technique as in (3), we may assume that $S^{r+2}H\otimes\det H^{\ast}$ is smooth without loss of generality. In particular, if we denote the weight function of $\Pi^{\ast}\det H^{\ast}$ by $\chi$, the function $\psi+\chi$ is also smooth.

Take the integral of
\[
i\partial\bar{\partial}\psi-\Pi^{\ast}c_{1}(\det E,\det H)
\]
along the fibre of $\Pi$ against the volume form induced by $i\partial\bar{\partial}_{v}\psi$, we then obtain a positive $(1,1)$-form $\Theta$ on $Y$. Apply the formula (\ref{e51}) on the vector bundle $S^{r+2}E\otimes\det E^{\ast}$, where the right hand side is exactly $\Theta$, the left hand side equals to
\[
c_{1}(\det S^{r+2}E,\det S^{r+2}H)-\textrm{R}\cdot c_{1}(\det E,\det H)=\frac{R}{r}c_{1}(\det E,\det H).
\]
This equality is due an elementary computation. Here $\textrm{R}=\binom{2r+2}{r+2}$ is the rank of $S^{r+2}E$. Therefore $c_{1}(\det E,\det H)$ is actually positive. Combine with the positivity of
\[
i\partial\bar{\partial}\varphi-\frac{1}{r+2}\pi^{\ast}c_{1}(\det E,\det H),
\]
we know that $i\partial\bar{\partial}\varphi$ is positive itself. Thus, $(E,H)$ is positive in the sense of Griffiths by definition.

(5) Notice that when $q=\frac{1}{2}$,
\[
\frac{2q-1}{r+1}=0.
\]
The conclusion is by definition.

(6) Using the same notations and regularization procedure as in (4), we may assume that $S^{r+2}H\otimes\det H^{\ast}$ is smooth, and
\[
(S^{r+2}E\otimes\det E^{\ast},S^{r+2}H\otimes\det H^{\ast})\]
is positive in the sense of Griffiths. Let $\Phi$ be the corresponding metric on $\mathcal{O}_{S^{r+2}E}(1)-\Pi^{\ast}\det E$, we obtain that
\[
(\mathcal{O}_{E}(r+2)-\pi^{\ast}\det E,i^{\ast}\Phi)
\]
is positive. Recall that $i:\mathbb{P}(E^{\ast})\hookrightarrow\mathbb{P}(S^{r+2}E^{\ast})$ is the Vernoese embedding, and $\Pi:\mathbb{P}(S^{r+2}E^{\ast})\rightarrow Y$, $\pi:X=\mathbb{P}(E^{\ast})\rightarrow Y$ are the natural projections. Now using the main theorem in \cite{Ber09}, we obtain that
\[
\pi_{\ast}(K_{X/Y}+\mathcal{O}_{E}(r+2)-\pi^{\ast}\det E)=E,
\]
equipped with the $L^{2}$-metric defined by $i^{\ast}\Phi$, is positive in the sense of Nakano. Obviously this metric is nothing but $H$.
\end{proof}

From now on, there is no need to differ the Griffiths positivity in our and Raufi's sense. However, we are willing to know that
\begin{problem}
Are the Nakano positivity and strongly Nakano positivity equivalent or not?
\end{problem}
If the answer is positive, we have successfully defined the Nakano positivity in the most general case.

\subsection{The vanishing theorem}
Now the generalised Griffiths' vanishing theorem is rather intuitive due to our discussions above.
\begin{proof}[Proof of Theorem \ref{t4}]
(1) Let $\varphi$ be the corresponding metric on $\mathcal{O}_{E}(1)$. Since $(E,H)$ is strictly positive in the sense of Griffiths, $(\mathcal{O}_{E}(1),\varphi)$ is big by definition. Thus, $\mathcal{O}_{E}(r+k+1)\otimes\pi^{\ast}L$ is also big for any $k$. Apply Nadel's vanishing theorem \cite{Nad90} to $\mathcal{O}_{E}(r+k+1)\otimes\pi^{\ast}L$, and remember that $\mathscr{I}((r+k+1)\varphi)=\mathcal{O}_{X}$ by assumption, we have
\[
H^{q}(X,K_{X}\otimes\mathcal{O}_{E}(r+k+1)\otimes\pi^{\ast}L)=0\textrm{ for all }q>0.
\]
On the other hand, observe that $\mathcal{O}_{E}(1)$ is ample along each fibre, using Kodaira's vanishing theorem  we obtain that
\[
R^{q}\pi_{\ast}(K_{X}\otimes\mathcal{O}_{E}(r+k+1)\otimes\pi^{\ast}L)=0\textrm{ for all }q>0.
\]
Now apply the Larey spectral sequence \cite{GrH78} we obtain that
\[
\begin{split}
0=&H^{q}(X,K_{X}\otimes\mathcal{O}_{E}(r+k+1)\otimes\pi^{\ast}L)\\
=&H^{q}(Y,\pi_{\ast}(K_{X}\otimes\mathcal{O}_{E}(r+k+1))\otimes L).
\end{split}
\]
Combining with the fact that
\[
\pi_{\ast}(K_{X}\otimes\mathcal{O}_{E}(r+k+1))=K_{Y}\otimes S^{k}E\otimes\det E,
\]
we immediately get the desired result.

(2) Let $\varphi$ be the corresponding metric on $\mathcal{O}_{E}(1)$, and let $\psi$ be the weight function of $h$. Since $(E,H)$ is positive in the sense of Griffiths, $(\mathcal{O}_{E}(1),\varphi)$ is pseudo-effective. In particular, $i\partial\bar{\partial}\varphi$ is strictly positive in the vertical direction. On the other hand, $(L,\psi)$ is big. Hence $(\pi^{\ast}L,\pi^{\ast}\psi)$ is pseudo-effective. In particular, $i\partial\bar{\partial}\pi^{\ast}\psi$ is strictly positive in the horizontal direction. As a result,
\[
i\partial\bar{\partial}((r+k+1)\varphi+\pi^{\ast}\psi)
\]
is strictly positive in every direction for any $k$. Therefore
\[
(\mathcal{O}_{E}(r+k+1)\otimes\pi^{\ast}L,(r+k+1)\varphi+\pi^{\ast}\psi)
\]
is big. We claim that $\mathscr{I}((r+k+1)\varphi+\pi^{\ast}\psi)=\mathscr{I}(\pi^{\ast}\psi)$.

One direction is obvious. Now for any point $x\in X$, we take a local coordinate ball $U$. If $f\in\mathscr{I}(\pi^{\ast}\psi)_{x}$, we have
\[
\begin{split}
\int_{U}|f|^{2}e^{-((r+k+1)\varphi+\pi^{\ast}\psi)}&\leqslant(\int_{U}(|f|^{2}e^{-\pi^{\ast}\psi})^{1+\varepsilon})^{\frac{1}{1+\varepsilon}} (\int_{U}e^{-\frac{(1+\varepsilon)(r+k+1)}{\varepsilon}\varphi})^{\frac{\varepsilon}{1+\varepsilon}}\\
&\leqslant C(\int_{U}(|f|^{2}e^{-\pi^{\ast}\psi})^{1+\varepsilon})^{\frac{1}{1+\varepsilon}}.
\end{split}
\]
The first inequality is due to H\"{o}lder's inequality and the second inequality comes from the fact that $\nu(\varphi)=0$. Notice that
\[
\int_{U}(|f|^{2}e^{-\pi^{\ast}\psi})^{1+\varepsilon}
\]
is finite for $\varepsilon$ small enough by openness property \cite{GuZ15}. It leads to the claim.

Now using the same argument as (1), we have
\[
H^{q}(X,K_{X}\otimes\mathcal{O}_{E}(r+k+1)\otimes\pi^{\ast}L\otimes\mathscr{I}(\pi^{\ast}\psi))=0\textrm{ for all }q>0
\]
and
\[
R^{q}\pi_{\ast}(K_{X}\otimes\mathcal{O}_{E}(r+k+1)\otimes\pi^{\ast}L\otimes\mathscr{I}(\pi^{\ast}\psi))=0\textrm{ for all }q>0.
\]
Then we actually have
\[
\begin{split}
0=&H^{q}(X,K_{X}\otimes\mathcal{O}_{E}(r+k+1)\otimes\pi^{\ast}L\otimes\mathscr{I}(\pi^{\ast}\psi))\\
=&H^{q}(Y,\pi_{\ast}(K_{X}\otimes\mathcal{O}_{E}(r+k+1))\otimes L\otimes\mathscr{I}(\psi))\\
=&H^{q}(Y,K_{Y}\otimes S^{k}E\otimes\det E\otimes L\otimes\mathscr{I}(\psi))
\end{split}
\]
for all $q>0$. Here we use the subadditivity \cite{DEL00} that
\[
\mathscr{I}(\pi^{\ast}\psi)=\pi^{\ast}\mathscr{I}(\psi)
\]
to get the second equality.

(3) Let $\psi$ be the weight function of $\det H$. Since $(E,H)$ is strictly strongly Nakano positive,
\[
(\mathcal{O}_{E}(r+2)\otimes\pi^{\ast}\det E^{\ast},(r+2)\varphi-\pi^{\ast}\psi)
\]
is big by definition. On the other hand, $(E,H)$ is strictly positive in the sense of Griffiths by (4) of Theorem \ref{t3}, $(\mathcal{O}_{E}(1),\varphi)$ is also big. Hence for every $k\geqslant1$
\[
\mathcal{O}_{E}(r+k+1)\otimes\pi^{\ast}\det E^{\ast}\otimes\pi^{\ast}L
\]
is big. Apply the same argument as before, we obtain that
\[
H^{q}(X,K_{X}\otimes\mathcal{O}_{E}(r+k+1)\otimes\pi^{\ast}\det E^{\ast}\otimes\pi^{\ast}L\otimes\mathscr{I}(\pi^{\ast}\psi))=0\textrm{ for all }q>0
\]
and
\[
R^{q}\pi_{\ast}(K_{X}\otimes\mathcal{O}_{E}(r+k+1)\otimes\pi^{\ast}\det E^{\ast}\otimes\pi^{\ast}L\otimes\mathscr{I}(\pi^{\ast}\psi))=0\textrm{ for all }q>0.
\]
Then we have
\[
\begin{split}
0=&H^{q}(X,K_{X}\otimes\mathcal{O}_{E}(r+k+1)\otimes\pi^{\ast}\det E^{\ast}\otimes\pi^{\ast}L\otimes\mathscr{I}(\pi^{\ast}\psi))\\
=&H^{q}(Y,\pi_{\ast}(K_{X}\otimes\mathcal{O}_{E}(r+k+1))\otimes\det E^{\ast}\otimes L\otimes\mathscr{I}(\psi))\\
=&H^{q}(Y,K_{Y}\otimes S^{k}E\otimes L\otimes\mathscr{I}(\psi))
\end{split}
\]
for all $q>0$.

(4) is similar with (2), and we omit the details here.
\end{proof}
In \cite{Ina20b}, there is also a generalised Griffiths' vanishing theorem (Corollary 1.4). Currently, the relationship between their work and Theorem \ref{t4} is not quite clearly to me. We present here the following example to show the partial inclusion. This example has been presented in \cite{Wu19} before.

\begin{example}[and Proposition 5.1]\label{ex51}
Let $E$ be a stable vector bundle of rank $r+1$ over a smooth projective curve $Y$ of genus $\geqslant2$. Let $k$ be a positive integer. Assume that $(\det E,\phi)$ is big with $\nu(\phi)<\frac{r+1}{k}$. Then
\[
H^{q}(Y,K_{Y}\otimes S^{k}E\otimes\det E)=0\textrm{ for all }q>0.
\]
\end{example}
One refers to \cite{Laz04} for the definition of a stable vector bundle. Notice that in this situation if we directly apply Corollary 1.4 in \cite{Ina20b}, we obtain that
\[
H^{q}(Y,K_{Y}\otimes E\otimes\det E)=0\textrm{ for all }q>0
\]
under the assumption that $\nu(\phi)<1$. Obviously $1<r+1$, hence it is more restricted than Proposition 5.1. We remark that the condition $\nu(\phi)<\frac{r+1}{k}$ in our result is not optimal yet.
\begin{proof}[Proof of Proposition 5.1]
Since $E$ is stable, it is proved in \cite{NaS65} that there exists a coordinate chart $\{V_{i}\}$ of $Y$ such that the transition matrices of $E$ can be written in the form $g_{ij}=f_{ij}U_{ij}$, where $f_{ij}$ is a scalar function and $U_{ij}$ is a unitary matrix on $V_{i}\cap V_{j}$. In this situation, there is a one-one correspondence between the singular Hermitian metrics on $E$ and $\det E$.

Let $h_{i}$ be the local representative on $V_{i}$ of the metric defined by $\phi$. Then $\{h^{1/(r+1)}_{i}I_{r+1}\}$ defines a singular Hermitian metric $H$ on $E$. Here $I_{r+1}$ is the identity matrix of rank $r+1$. The associated curvature of $H$ is then given by
\[
\frac{i}{r+1}h^{1/(r+1)}_{i}\partial\bar{\partial}\phi\cdot I_{r+1},
\]
which is strictly positive in the sense of Griffiths. Let $\varphi$ be the metric on $\mathcal{O}_{E}(1)$ induced by $H$. Then $(\mathcal{O}_{E}(1),\varphi)$ is big. Moreover, recall that for any positive integer $k$, if $R=\binom{r+k+1}{k}$ is the rank of $S^{k}E$, we have $\det S^{k}E=\frac{kR}{r+1}\det E$. Thus, $\{h_{i}\}$ defines a metric $\{h^{kR/(r+1)}_{i}\}$ on $\det S^{k}E$ and $\{h^{k/(r+1)}_{i}I_{r+1}\}$ on $S^{k}E$. It is easy to verify that the later one coincides with $S^{k}H$.

Since $(E,H)$ is strictly positive in the sense of Griffiths,
\[
H^{q}(Y,K_{Y}\otimes S^{k}E\otimes\det E)=0\textrm{ for }q>0
\]
once we have $\mathscr{I}((r+k+1)\varphi)=\mathcal{O}_{X}$.
So it is left to prove that $\mathscr{I}((r+k+1)\varphi)=\mathcal{O}_{X}$.

Remember that in the proof of Theorem \ref{t3}, we introduce the Vernoese embedding $i:\mathbb{P}(E^{\ast})\hookrightarrow\mathbb{P}(S^{k}E^{\ast})$. Let $\psi$ be the metric on $\mathcal{O}_{S^{k}E}(1)$ corresponds to $S^{k}H$. In this context, we have
\[
k\varphi=i^{\ast}\psi.
\]
On the other hand, let
\[
\begin{split}
\Pi:\mathcal{X}=\mathbb{P}(S^{k}E^{\ast})&\rightarrow Y\\
(z,W)&\mapsto z
\end{split}
\]
be the projection, let $\Xi$ be an arbitrary section of $\Pi_{\ast}\mathcal{O}_{S^{k}E}(1)$ and let $\{U_{\alpha}\}$ be the corresponding section of $S^{k}E$ under the canonical isomorphism:
\[
\Pi_{\ast}\mathcal{O}_{S^{k}E}(1)\simeq S^{k}E.
\]
Under the same computation as Theorem \ref{t1}, we have
\begin{equation}\label{e52}
    \int_{\mathcal{X}_{z}}|\Xi|^{2}e^{-\psi}\frac{\omega^{R-1}_{\psi,z}}{(R-1)!}=h^{\frac{k}{r+1}}_{i}\sum|U_{\alpha}|^{2}
\end{equation}
up to a constant.

Since by assumption we have $\nu(h_{i})<\frac{r+1}{k}$, $h^{\frac{k}{r+1}}_{i}$ is locally integrable on $Y$ by \cite{Sko72}. Therefore the function
\[
\int_{\mathcal{X}_{z}}|\Xi|^{2}e^{-\psi}\frac{\omega^{R-1}_{\psi,z}}{(R-1)!}
\]
is also locally integrable by formula (\ref{e52}). Observe that $\psi$ is induced by a Hermitian metric $S^{k}H$ on $S^{k}E$ hence solves the equation (\ref{e41}) on $\mathcal{X}$. If we locally expand
\[
e^{-\psi}\omega^{R-1}_{\psi,z}=e^{-(R+1)\psi}\Pi^{\ast}\det S^{k}HdW\wedge d\bar{W},
\]
the function
\[
e^{-(R+1)\psi}\Pi^{\ast}\det S^{k}H
\]
is locally integrable on $\mathcal{X}$ by Fubini's theorem.
This integrability keeps after pull-back through $i$. Let
\[
\begin{split}
\pi:X=\mathbb{P}(E^{\ast})&\rightarrow Y\\
(z,w)&\mapsto z
\end{split}
\]
be the projection. Thus
\[
\begin{split}
i^{\ast}(e^{-(R+1)\psi}\Pi^{\ast}\det S^{k}H)=e^{-(R+1)k\varphi}\pi^{\ast}\det S^{k}H
\end{split}
\]
is locally integrable on $X$. Since
\[
\begin{split}
(R+1)k&=(\binom{r+k+1}{k}+1)k\\
&\geqslant r+k+1,
\end{split}
\]
it implies the local integrability of $e^{-(r+k+1)\varphi}$ on $X$. Therefore
\[
\mathscr{I}((r+k+1)\varphi)=\mathcal{O}_{X}.
\]
The proof is complete.
\end{proof}

\address{

\small Current address: School of Mathematical Sciences, Fudan University, Shanghai 200433, People's Republic of China.

\small E-mail address: jingcaowu13@fudan.edu.cn
}


\begin{thebibliography}{99}


\bibitem{Ber09}
Berndtsson, B.: Curvature of vector bundles associated to holomorphic fibra-tions. Ann. of Math. \textbf{169}, 531-560 (2009)

\bibitem{BP08}
Berndtsson, B., P\u{a}un, M.: Bergman kernels and the pseudoeffectivity of relative canonical bundles. Duke Math. J. \textbf{145}, 341-378 (2008)

\bibitem{BeT82}
Bedford, E., Taylor, B. A.: A new capacity for plurisubharmonic functions. Acta Math. \textbf{149}, 1-41 (1982)

\bibitem{Cat98}
de Cataldo, M.: Singular Hermitian metrics on vector bundles. J. Reine Angew. Math. \textbf{502}, 93-122 (1998)

\bibitem{CTi08}
Chen, X. X., Tian, G.: Geometry of K\"{a}hler metrics and foliations by holomorphic discs. Publ. Math. Inst. Hautes \'{E}tudes \textbf{107}, 1-107 (2008)

\bibitem{Dem92a}
Demailly, J.-P.: Singular Hermitian metrics on positive line bundles. Lecture Notes in Math. \textbf{1507}, Springer, Berlin, 87-104 (1992)

\bibitem{Dem92b}
Demailly, J.-P.: Regularization of closed positive currents and intersection theory. J. Algebraic Geom. \textbf{1}, 361-409 (1992)

\bibitem{Dem12}
Demailly, J.-P.: Analytic methods in algebraic geometry. Surveys of Modern Mathematics 1, International Press, Somerville, MA; Higher Education Press, Beijing (2012)

\bibitem{DEL00}
Demailly, J.-P., Ein, L., Lazarsfeld, R.: A subadditivity property of multiplier ideals. Michigan Math. J. \textbf{48}, 137-156 (2000)

\bibitem{DPS01}
Demailly, J.-P., Peternell, T., Schneider, M.: Pseudo-effective line bundles on compact K\"{a}hler manifolds. Internat. J. Math. \textbf{6}, 689-741 (2001)

\bibitem{DSk79}
Demailly, J.-P., Skoda, H.: Relations entre les notions de positivit\'{e} de P.A. Griffiths et de S. Nakano, S\'{e}minaire P. Lelong-H. Skoda (Analyse), ann\'{e}e 1978/79, Lecture notes in Math. no 822, Springer-Verlag, Berlin, 304-309 (1980)

\bibitem{Don01}
Donaldson, S.: Scalar curvature and projective embeddings. I. J. Diff. Geom. \textbf{59}, 479-522 (2001)

\bibitem{Don05}
Donaldson, S.: Scalar curvature and projective embeddings. II. Q. J. Math. \textbf{56}, 345-356 (2005)

\bibitem{GrH78}
Griffiths, P., Harris, J.: Principles of algebraic geometry, Wiley, New York (1978)

\bibitem{Gri69}
Griffiths, P.: Hermitian differential geometry, Chern classes and positive vector bundles. Global Analysis, Papers in Honor of K. Kodaira, Princeton University Press, Princeton, NJ (1969)

\bibitem{GuZ15}
Guan, Q., Zhou, X.: A proof of Demailly's strong openness conjecture. Ann. of Math. (2) \textbf{182}, 605-616 (2015)

\bibitem{Har77}
Hartshorne, R.: Algebraic geometry. Graduate Texts in Mathematics, No. 52. Springer-Verlag, New York-Heidelberg, 1977. xvi+496 pp. ISBN: 0-387-90244-9

\bibitem{Har80}
Hartshorne, R.: Stable reflexive sheaves. Math. Ann. \textbf{254}, 121-176 (1980)

\bibitem{Hos17}
Hosono, G.: Approximations and examples of singular Hermitian metrics on vector bundles. Ark. Mat. \textbf{55}, 131-153 (2017)

\bibitem{HoI20}
Hosono, G., Inayama, T.: A converse of H\"{o}rmander's $L^{2}$-estimate and new positivity notions for vector bundles. Science China Mathematics. https://doi.org/10.1007/s11425-019-1654-9.

\bibitem{Ina20a}
Inayama, T.: Curvature Currents and Chern Forms of Singular Hermitian Metrics on Holomorphic Vector Bundles. J. Geom. Anal. \textbf{30}, 910-935 (2020)

\bibitem{Ina20b}
Inayama, T.: $L^{2}$ Estimates and Vanishing Theorems for Holomorphic Vector Bundles Equipped with Singular Hermitian Metrics. Michigan Math. J. \textbf{69}, 79-96 (2020)

\bibitem{Ina20c}
Inayama, T.: Nakano positivity of singular Hermitian metrics and vanishing theorems of Demailly--Nadel--Nakano type. arXiv:2004.05798.

\bibitem{Kob75}
Kobayashi, S.: Negative vector bundles and complex Finsler structures. Nagoya Math. J. \textbf{57}, 153-166 (1975)

\bibitem{Kob87}
Kobayashi, S.: Differential geometry of complex vector bundles. Publications of the Mathematical Society of Japan, 15. Kan\^{o} Memorial Lectures, 5. Princeton University Press, Princeton, NJ; Princeton University Press, Princeton, NJ, 1987. xii+305 pp. ISBN: 0-691-08467-X

\bibitem{Kob96}
Kobayashi, S.: Complex Finsler vector bundles. Contemp. Math. \textbf{196}, Amer. Math. Soc., Providence,
RI, 133-144 (1996)

\bibitem{Laz04}
Lazarsfeld, R.: Positivity in algebraic geometry. II. Positivity for vector bundles, and multiplier ideals., \textbf{49}, Berlin: Springer (2004)

\bibitem{LRRS18}
L\"{a}rk\"{a}ng, R., Raufi, H., Ruppenthal, J., Sera, M.: Chern forms of singular metrics on vector bundles. Advances in Mathematics \textbf{326}, 465-489 (2018)

\bibitem{LSY06}
Liu, K., Sun, X., Yang, X.: Positivity and vanishing theorems for ample vectorbundles, J. Algebraic Geom. \textbf{72}, 303-331 (2006)

\bibitem{Man97}
Manivel, L.: Vanishing theorems for ample vector bundles. Invent. math. \textbf{127}, 401-416 (1997)

\bibitem{Nad90}
Nadel, Alan M.: Multiplier ideal sheaves and K\"{a}hler--Einstein metrics of positive scalar curvature. Ann. of Math. (2) \textbf{132}, 549-596 (1990)

\bibitem{Nak55}
Nakano, S.: On complex analytic vector bundles. J. Math. Soc. Japan. \textbf{7}, 1-12 (1955)

\bibitem{NaS65}
Narasimhan, M. S., Seshadri, C. S.: Stable and unitary vector bundles on compact Riemann suefaces. Ann. of Math. \textbf{82}, 540-567 (1965)

\bibitem{Nau17}
Naumann, P.:  An approach to Griffiths conjecture. arXiv:1710.10034v1

\bibitem{PTa18}
P\u{a}un, M., Takayama, S.: Positivity of twisted relative pluricanonical bundles and their direct images. J. Algebraic Geom. \textbf{27}, 211-272 (2018)

\bibitem{Rau15}
Raufi, H.: Singular hermitian metrics on holomorphic vector bundles. Ark. Mat. \textbf{53}, 359-382 (2015)

\bibitem{Siu98}
Siu, Y.-T.: Invariance of plurigenera. Invent. Math. \textbf{143}, 661-673 (1998)

\bibitem{Sko72}
Skoda, H.: Sous-ensembles analytiques d'ordre fini ou infini dans $\mathbb{C}^{n}$. Bull. Soc. Math. France \textbf{100}, 353-408 (1972)

\bibitem{Wu19}
Wu, J.: Nefness of the direct images of relative canonical bundles. arXiv:1906.07885.


\end{thebibliography}
\end{document}